\documentclass[12pt]{amsart}
\usepackage{amscd}
\usepackage{amsmath}
\usepackage{amssymb}
\usepackage{amsthm}
\newtheorem{theorem}{Theorem}[section]

\newtheorem{lemma}[theorem]{Lemma}
\newtheorem{proposition}[theorem]{Proposition}

\newtheorem{corollary}[theorem]{Corollary}

\theoremstyle{definition}

\newtheorem{example}[theorem]{Example}

\newtheorem*{acknowledgement}{Acknowledgments}
\theoremstyle{remark}
\newtheorem{remark}[theorem]{Remark}
\numberwithin{equation}{section}

\DeclareMathOperator{\Tor}{Tor}
\DeclareMathOperator{\pd}{pd}
\DeclareMathOperator{\Ker}{Ker}
\DeclareMathOperator{\Image}{Im}

\DeclareMathOperator{\reg}{reg}

\DeclareMathOperator{\lcm}{lcm}

\begin{document}
\title[Non-vanishing of Betti numbers of edge ideals]
  {Non-vanishing of Betti numbers of edge ideals and complete bipartite subgraphs}
\author[Kyouko Kimura]{Kyouko Kimura}
\address{Department of Mathematics, Graduate School of Science, 
          Shizuoka University, 836 Ohya, Suruga-ku, Shizuoka 422-8529, Japan}
\email{skkimur@ipc.shizuoka.ac.jp}
\date{}
\begin{abstract}
  Given a finite simple graph one can associate the edge ideal. 
  In this paper we prove that a graded Betti number of the edge ideal 
  does not vanish if the original graph contains a set of complete bipartite 
  subgraphs with some conditions. 
  Also we give a combinatorial description for the projective dimension 
  of the edge ideals of unmixed bipartite graphs. 
\end{abstract}
\maketitle

\section{Introduction}
Let $G$ be a finite simple graph, i.e., 
a finite graph with no loop and no multiple edge. 
We denote by $V=V(G)$, the vertex set of $G$, 
and by $E(G)$, the edge set of $G$. Put $N = \# V$. 
Let $S := K[V]$ be the polynomial ring over a field $K$ 
whose variables are vertices of $G$. 
We consider the standard ${\mathbb{N}}^N$-grading on $S$ 
unless otherwise specified. 
With a graph $G$, we can associate a quadratic squarefree monomial ideal 
$I(G)$ of $S$: 
\begin{displaymath}
  I(G) = (u v \; : \; \{ u,v \} \in E(G)), 
\end{displaymath}
which is called the \textit{edge ideal} of $G$. 
The study of edge ideals was inspired by Villarreal \cite{Villarreal90}. 
The main direction of the study of edge ideals is to investigate 
relations between algebraic properties of edge ideals and 
combinatorial ones of original graphs; 
see \cite{HaTuylsurvey, MV, Villarreal} and their references. 
We are interested in invariants arising from a minimal free resolution. 
The (Castelnuovo--Mumford) regularity and the projective dimension are 
examples of such invariants. 
Many authors investigate these invariants. 
Especially, in 
\cite{FHVT, HaTuyl08, KAM, Kummini, MMCRTY, VanTuyl, Woodroofe10, Zheng}, 
the regularity of edge ideals of some classes of the edge ideals 
were characterized with the notion of the $3$-disjointness of edges; 
two edges $e_1, e_2 \in E(G)$ are said to be \textit{$3$-disjoint} if 
$e_1 \cap e_2 = \emptyset$ and the edge set of the induced subgraph of $G$ 
on $e_1 \cup e_2$ is $\{ e_1, e_2 \}$; 
see Section \ref{sec:pre} for more detail. 

\par
In this paper, we focus on graded Betti numbers. 
Notice that the regularity as well as the projective dimension 
is defined via graded Betti numbers. 
Although some formulas or estimations of ($\mathbb{N}$-)graded Betti numbers 
of edge ideals are discovered in e.g., 
\cite{Bouchat, Chen, CN, DE, EV, FRG09, Goff, HaTuyl07, Jacques, JacqKatz, Katzman, Kimura, RVT, Visscher}, 
it seems that the combinatorial meanings of these are not known so much. 
In \cite{Kimura}, the author explored this direction 
following Katzman \cite{Katzman}. 
In particular, the author gave a non-vanishing theorem 
\cite[Theorem 3.1]{Kimura} of graded Betti numbers of edge ideals 
which was an improvement of Katzman's one \cite[Proposition 2.5]{Katzman}; 
see Section \ref{sec:pre} for more details. 
In this paper, we improve \cite[Theorem 3.1]{Kimura} much more. 
The following theorem is the main result of the paper. 
\begin{theorem}
  \label{MainIntro}
  Let $G$ be a finite simple graph on the vertex set $V$. 
  Suppose that there exists a set of complete bipartite subgraphs 
  $\{ B_1, \ldots, B_r \}$ $(r \geq 1)$ of $G$ 
  satisfying the following $2$ conditions: 
  \begin{enumerate}
  \item $V(B_k) \cap V(B_{\ell}) 
         = \emptyset$ for all $k \neq \ell$. 
  \item There exist edges $e_1, \ldots, e_r$ 
    with $e_k \in E(B_k)$, $k=1, \ldots, r$, 
    which are pairwise $3$-disjoint in $G$. 
  \end{enumerate}
  Set 
  \begin{displaymath}
    \sigma = V(B_1) \cup \cdots \cup V(B_r). 
  \end{displaymath}
  Then identifying $\sigma$ with 
  the $(0,1)$-vector $({\sigma}_v) \in {\mathbb{N}}^{N}$ 
  such that ${\sigma}_v = 1$ if and only if $v \in \sigma$, 
  we have
  \begin{displaymath}
    \beta_{|{\sigma}| - r, \sigma} (S/I(G)) \neq 0. 
  \end{displaymath}
\end{theorem}
Actually, in \cite[Theorem 3.1]{Kimura}, we restricted each $B_k$ 
to be a bouquet, that is a complete bipartite subgraph of type $(1, n_k)$. 
However in the theorem, 
we allow $B_k$ to be an arbitrary complete bipartite subgraph. 
Also the condition (2) of Theorem \ref{MainIntro}, 
which is the same one as in \cite[Theorem 3.1]{Kimura}, 
is not so peculiar because of the relation between 
the regularity and the $3$-disjointness of edges; see Section \ref{sec:pre} 
for more details. 

\par
\bigskip

\par
Now we explain the organization of the paper. 
In the next section, we recall some notions and 
observe some known results 
about relations between the regularity of edge ideals 
and the $3$-disjointness of the original graphs. 
Then the meaning of the condition (2) of Theorem \ref{MainIntro} will be clear. 
Also we recall preceding results: \cite[Proposition 2.5]{Katzman} and 
\cite[Theorem 3.1]{Kimura}. 
In Section \ref{sec:Lyu-res}, 
we recall the definition of a Lyubeznik resolution (\cite{Ly}) 
and show some properties of it, 
which plays a key role in the proof of Theorem \ref{MainIntro}. 
Then in Section \ref{sec:proof}, we prove Theorem \ref{MainIntro}. 

\par
Theorem \ref{MainIntro} gives a sufficient condition for non-vanishing 
to the graded Betti numbers of edge ideals. 
One would be interested in 
how near the condition is to a necessary one. 
In \cite{Kimura}, the author proved that the condition of 
\cite[Theorem 3.1]{Kimura} (and thus that of Theorem \ref{MainIntro}) 
is a necessary condition for edge ideals of chordal graphs 
(\cite[Theorem 4.1]{Kimura}, actually, the author proved this only for 
the $\mathbb{N}$-graded case, but this is still true 
for the ${\mathbb{N}}^N$-graded case). 
In Sections \ref{sec:LinearStrand} and \ref{sec:CMbipartite}, 
we give other partial answers to the above question; 
Propositions  \ref{LinearStrand} and \ref{extrem-CMbipartite}. 
In the argument on Sections \ref{sec:LinearStrand} and \ref{sec:CMbipartite}, 
complete bipartite subgraphs naturally arise. 
Moreover we give combinatorial descriptions for 
the projective dimension of the edge ideal of 
a co-chordal graph (Corollary \ref{co-chordalProj}) 
and a Cohen--Macaulay bipartite graph (Corollary \ref{CMbipartiteProj}). 
Finally in Section \ref{sec:UnmixedBipartite}, 
we treat an unmixed bipartite graph and 
give a combinatorial description for the projective dimension of the edge 
ideal of it (Theorem \ref{unmixedProj}) 
as a generalization of the result for a Cohen--Macaulay bipartite graph 
by using Kummini's consideration \cite{Kummini}. 

\section{Preliminaries and known results}
\label{sec:pre}
In this section, we recall some notions 
and some known results about the regularity of edge ideals. 

\par
Let $S$ be a polynomial ring in $N$ variables over a field $K$ 
with the ${\mathbb{N}}^N$-grading 
and $M$ a ${\mathbb{N}}^N$-graded $S$-module. 
Let 
\begin{displaymath}
  0 \longrightarrow \bigoplus_{\sigma \in {\mathbb{N}}^N} 
                      S(-\sigma)^{\beta_{p, \sigma}}
    \longrightarrow \cdots 
    \longrightarrow \bigoplus_{\sigma \in {\mathbb{N}}^N} 
                      S(-\sigma)^{\beta_{1, \sigma}} 
    \longrightarrow \bigoplus_{\sigma \in {\mathbb{N}}^N} 
                      S(-\sigma)^{\beta_{0, \sigma}} 
    \longrightarrow M \longrightarrow 0
\end{displaymath}
be a minimal ${\mathbb{N}}^N$-graded free resolution of $M$. 
The \textit{projective dimension} of $M$, 
denoted by $\pd (M)$ is the length $p$ of the resolution 
and ${\beta}_{i, \sigma} (M) := {\beta}_{i, \sigma}$ is 
the $i$th ${\mathbb{N}}^N$-graded Betti number of $M$ of degree $\sigma$. 
In terms of Betti numbers, 
the projective dimension of $M$ is described as follows: 
\begin{displaymath}
  \pd (M) = \max \{ i \; : \; \beta_{i, \sigma} (M) \neq 0 \}. 
\end{displaymath}
Also the \textit{(Castelnuovo--Mumford) regularity} of $M$ is defined by 
\begin{displaymath}
  \reg (M) := \max \{ |{\sigma}| -i \; : \; \beta_{i,\sigma} (M) \neq 0 \}, 
\end{displaymath}
where $|{\sigma}| = \sigma_{1} + \cdots + \sigma_{N}$ for 
$\sigma = ({\sigma}_{\ell}) \in {\mathbb{N}}^N$. 
We can also consider the $\mathbb{N}$-grading on $S$ as well as on $M$. 
Then the $i$th $\mathbb{N}$-graded Betti number $\beta_{i,j} (M)$ of $M$ of 
degree $j \in \mathbb{N}$ is defined similarly to $\beta_{i, \sigma} (M)$. 
Note that 
$\beta_{i,j} (M) 
 = \sum_{\sigma \in {\mathbb{N}}^N, \  |{\sigma}| = j} \beta_{i, \sigma} (M)$. 

\par
Let $d$ be the minimum degree among non-zero elements in $M$ 
in $\mathbb{N}$-grading. 
Then $\beta_{i, \sigma} (M)$ with $|{\sigma}| = i+d$ is called a graded 
Betti number in the \textit{linear strand} of $M$. 
When each graded Betti number of $M$ which is not in the linear strand 
vanishes, we say that $M$ has a \textit{linear resolution}. 

\par
Let $G$ be a finite simple graph on the vertex set $V$. 
Set $N = \# V$. 
Since the edge ideal $I(G)$ is a squarefree monomial ideal, 
the Hochster's formula
for Betti numbers (\cite[Theorem 5.5.1]{BrunsHerzog}) shows that 
$\beta_{i, \sigma} (S/I(G)) = 0$ unless $\sigma \in {\mathbb{N}}^N$ 
is a $(0,1)$-vector. 
In what follows, we identify the subset $\sigma$ of $V$ with 
the $(0,1)$-vector $({\sigma}_v) \in {\mathbb{N}}^{N}$ 
defined by ${\sigma}_v = 1$ iff $v \in \sigma$. 

\par
A graph $G$ on the vertex set $V$ is called \textit{bipartite} 
if $V$ can be bipartitioned as 
$V = V_1 \sqcup V_2$ with $\{ v, v' \} \notin E(G)$ for all $v, v' \in V_i$ 
($i=1, 2$). 
Let $G$ be a bipartite graph with a bipartition $V = V_1 \sqcup V_2$. 
If $\{v_1, v_2 \} \in E(G)$ for all $v_1 \in V_1$ and for all $v_2 \in V_2$, 
then the graph $G$ is called a \textit{complete bipartite graph 
of type $(\# V_1, \# V_2)$}. 
In particular, a complete bipartite graph of type $(1, n)$ 
is called a \textit{bouquet}. 
Let $c \geq 2$ be an integer. 
A graph $G$ on the vertex set $V$ is 
a \textit{complete $c$-partite graph} if $V$ can be partitioned as 
$V = V_1 \sqcup V_2 \sqcup \cdots \sqcup V_c$ 
with the property: for $u \in V_{k_1}$ and $v \in V_{k_2}$, 
$\{ u, v \} \in E(G)$ 
if and only if $k_1 \neq k_2$. 

\par
A graph is called a \textit{forest} if it has no cycle. 
A \textit{chordal graph} is a graph whose cycle of length $>3$ has a chord. 
We say that a graph $G$ is Cohen--Macaulay (over $K$) if 
$S/I(G)$ is Cohen--Macaulay. 
Also we say that a graph $G$ is unmixed if 
$I(G)$ is height unmixed. 
It is well known that if $G$ is Cohen--Macaulay, then 
$G$ is unmixed. 

\par
Let $G$ be a finite simple graph on the vertex set $V$. 
For a subset $\sigma \subset V$, we denote by $G_{\sigma}$, 
the \textit{induced subgraph} of $G$ on $\sigma$. 
That is, $G_{\sigma}$ is the subgraph of $G$ whose vertex set is $\sigma$ and 
whose edge set consists of all edges in $E(G)$ contained in $\sigma$. 
A subgraph $G'$ of $G$ is called a \textit{spanning subgraph} of $G$ if 
$V(G') = V(G)$. 
Let $e_1, e_2$ be two edges of $G$. 
Then we say that $e_1$ and $e_2$ are \textit{$3$-disjoint} if 
$e_1 \cap e_2 = \emptyset$ and $E(G_{e_1 \cup e_2}) = \{ e_1, e_2 \}$. 
A subset $\{ e_1, \ldots, e_r \} \subset E(G)$ is called 
\textit{pairwise $3$-disjoint} if $e_k$ and $e_{\ell}$ are $3$-disjoint 
for all $k \neq \ell$. 
We denote by $a(G)$, the maximum cardinality of pairwise $3$-disjoint subsets 
of edges of $G$. 
The invariant $a(G)$ is closely related to the regularity of $G$. 
Actually, Katzman \cite{Katzman} proved the following theorem. 
\begin{theorem}[{\cite[Lemma 2.2]{Katzman}}]
  \label{reg>=a(G)}
  For a finite simple graph $G$, we have 
  \begin{displaymath}
    \reg (S/I(G)) \geq a(G). 
  \end{displaymath}
\end{theorem}
Zheng \cite{Zheng} proved that the equality 
$\reg (S/I(G)) = a(G)$ holds when $G$ is a forest. 
Later, many authors 
\cite{FHVT, HaTuyl08, KAM, Kummini, MMCRTY, VanTuyl, Woodroofe10} 
discovered classes of graphs with the equality. 
Among these results, we note here the following graphs. 
\begin{theorem}[{H\`{a} and Van Tuyl \cite[Theorem 6.8]{HaTuyl08}, 
  Kummini \cite[Theorem 1.1]{Kummini}}]
  \label{reg=a(G)}
  For the following classes of graphs $G$, we have $\reg (S/I(G)) = a(G)$: 
  \begin{enumerate}
  \item chordal graphs. 
  \item unmixed bipartite graphs, especially, Cohen--Macaulay bipartite graphs. 
  \end{enumerate}
\end{theorem}
These theorems implies the naturalness of the condition (2) 
of Theorem \ref{MainIntro}. 

\par
We note that in general, the equality $\reg (S/I(G)) = a(G)$ does not hold. 
For example, 
Kummini \cite{Kummini} noted that for the octahedron $C_8$, 
which is a bipartite graph, 
the proper inequality $\reg (S/I(C_8)) > a(C_8)$ holds. 
A class of graphs $G$ those satisfy $\reg (S/I(G)) > a(G)$ is also found in 
Nevo \cite{Nevo}. 

\par
\bigskip

\par
We close this section by observing preceding results. 
First we recall Katzman's result \cite[Propsition]{Katzman}. 
It can be rewrite as the following form. 
\begin{theorem}[{Katzman \cite{Katzman}}]
  \label{KatzmanBetti}
  Let $G$ be a finite simple graph on the vertex set $V$ and 
  $\sigma$ a subset of $V$. 
  If $G_{\sigma}$ is the disjoint union of $r$ bouquets, 
  then 
  \begin{displaymath}
    \beta_{|{\sigma}| - r, |{\sigma}|} (S/I(G)) \neq 0. 
  \end{displaymath}
\end{theorem}
On Theorem \ref{KatzmanBetti}, any two edges of $G_{\sigma}$ 
which belong to different bouquets are $3$-disjoint. 
The author \cite{Kimura} improved this as the following form. 
\begin{theorem}[{\cite[Theorem 3.1]{Kimura}}]
  \label{KimuraBetti}
  Let $G$ be a finite simple graph on the vertex set $V$. 
  Suppose that there exists a set of bouquets 
  $\{ B_1, \ldots, B_r \}$ $(r \geq 1)$ of $G$ 
  satisfying the following $2$ conditions: 
  \begin{enumerate}
  \item $V(B_k) \cap V(B_{\ell}) 
         = \emptyset$ for all $k \neq \ell$. 
  \item There exist edges $e_1, \ldots, e_r$ 
    with $e_k \in E(B_k)$, $k=1, \ldots, r$, 
    which are pairwise $3$-disjoint in $G$. 
  \end{enumerate}
  Set $\sigma = V(B_1) \cup \cdots \cup V(B_r)$. 
  Then we have
  \begin{displaymath}
    \beta_{|{\sigma}| - r, \sigma} (S/I(G)) \neq 0. 
  \end{displaymath}
\end{theorem}
Theorem \ref{MainIntro} is a further improvement of the above results; 
it says that the claim is still true if we replace bouquets with 
arbitrary complete bipartite subgraphs on Theorem \ref{KimuraBetti}.

\section{Lyubeznik resolutions}
\label{sec:Lyu-res}
Let $S$ be a polynomial ring over a field $K$. 
For a monomial ideal $I \subset S$, 
the explicit free resolution of $S/I$, 
so-called Taylor resolution, is known. 
A Lyubeznik resolution is a subcomplex of Taylor resolution. 
Although both resolutions are not minimal in general, 
a Lyubeznik resolution is quite nearer to minimal than Taylor resolution. 
In this section, we prove sufficient conditions for non-vanishing 
to graded Betti numbers of monomial ideals, 
which are derived from Lyubeznik resolutions and 
indispensable for our proof of Theorem \ref{MainIntro}. 

\par
We first recall the construction of Taylor resolution 
and Lyubeznik resolutions. 
Taylor resolution $(T_{\bullet}, d_{\bullet})$ of $I$ is defined as follows. 
Let $\mathcal{G} (I) = \{ m_1, \ldots, m_{\mu} \}$ be the minimal set of 
monomial generators of $I$. 
Let $s \geq 1$ be an integer. 
For $1 \leq {\ell}_1 < \cdots < {\ell}_{s} \leq \mu$, 
let us consider the symbol $e_{{\ell}_1 \cdots {\ell}_s}$ whose degree is 
$\deg \lcm (m_{{\ell}_1}, \ldots, m_{{\ell}_s})$. 
Let $T_{s}$ be the free $S$-module generated by all 
$e_{{\ell}_1 \cdots {\ell}_s}$, 
$1 \leq {\ell}_1 < \cdots < {\ell}_s \leq \mu$. 
The differential map $d_{s}$ is given by 
\begin{displaymath}
  d_s (e_{{\ell}_1 \cdots {\ell}_s}) 
  = \sum_{t=1}^{s} (-1)^{t - 1} 
    \frac{\lcm (m_{{\ell}_1}, \ldots, m_{{\ell}_s})}{\lcm (m_{{\ell}_1}, \ldots, \widehat{m_{{\ell}_t}}, \ldots, m_{{\ell}_s})} 
    e_{{\ell}_1 \cdots \widehat{{\ell}_t} \cdots {\ell}_s}. 
\end{displaymath}
It is in fact a free resolution of $S/I$ 
(see e.g., \cite[Exercise 17.11]{Eisenbud}). 

\par
A Lyubeznik resolution (\cite{Ly}) is a subcomplex of Taylor resolution. 
To construct a Lyubeznik resolution of $I$, we first fix an order of the 
minimal monomial generators of $I$: $m_1, \ldots, m_{\mu}$. 
The symbol $e_{{\ell}_1 \cdots {\ell}_s}$ is said to be $L$-admissible if 
for all $1 \leq t < s$ and for all $q < {\ell}_t$, 
$\lcm (m_{{\ell}_t}, \ldots, m_{{\ell}_s})$ is not divisible by $m_q$. 
Then the Lyubeznik resolution of $I$ with respect to the above order 
of the minimal monomial generators is 
a subcomplex of Taylor resolution generated by all $L$-admissible symbols; 
it is in fact a free resolution of $S/I$ (Lyubeznik \cite{Ly}). 
Since $L$-admissibleness depends on an order of monomial generators, 
a Lyubeznik resolution also depends on it. 

\par
Although a Lyubeznik resolution is not minimal in general, 
we can obtain some non-vanishment of graded Betti numbers 
of $S/I$ from it. 
For example, Barile \cite[Remark 1]{Barile05} note that 
$\beta_{s, \sigma} (S/I) \neq 0$ if there exists 
a maximal $L$-admissible symbol $e_{{\ell}_1 \cdots {\ell}_s}$ 
with $\deg e_{{\ell}_1 \cdots {\ell}_s} = \sigma$ satisfying 
\begin{displaymath}
  \deg e_{{\ell}_1 \cdots \widehat{{\ell}_t} \cdots {\ell}_s} 
    \neq \deg e_{{\ell}_1 \cdots {\ell}_s}, 
  \qquad t = 1, 2, \ldots, s.  
\end{displaymath}
Here we say that an $L$-admissible symbol $e_{{\ell}_1 \cdots {\ell}_s}$ 
is maximal if $e_{k_1 \cdots k_{s'}}$ is not $L$-admissible whenever 
$\{ {\ell}_1, \ldots, {\ell}_s \} \subsetneq \{ k_1, \ldots, k_{s'} \}$. 

\par
We obtain another non-vanishment of graded Betti numbers of $S/I$ 
by observing a Lyubeznik resolution more detailed. 
\begin{proposition}
  \label{Lyu-resBetti}
  Let $S = K[x_1, \ldots, x_N]$ be a polynomial ring over a field $K$ 
  and $I$ a monomial ideal of $S$. 
  Set $\mathcal{G} (I) = \{ m_1, \ldots, m_{\mu} \}$. 
  Let $L_{\bullet}$ be the Lyubeznik resolution of $I$ with 
  the above order of generators. 
  Suppose that there exists a maximal $L$-admissible symbol 
  $e_{{\ell}_1 \cdots {\ell}_{s}} \in L_{s}$, 
  $L$-admissible symbols 
  $e_{{\ell}_1' \cdots {\ell}_{s}'} \in L_{s}$ 
  ($e_{{\ell}_1' \cdots {\ell}_{s}'} 
     \neq e_{{\ell}_1 \cdots {\ell}_{s}}$), 
  and elements $a_{{\ell}_1' \cdots {\ell}_{s}'} \in S$ such that
  \begin{equation}
    \label{eq:max}
    d_{s} ( e_{{\ell}_1 \cdots {\ell}_{s}}
     +\sum_{\{ {\ell}_1', \ldots, {\ell}_{s}' \} 
            \neq \{ {\ell}_1, \ldots, {\ell}_s \}} 
      a_{{\ell}_1' \cdots {\ell}_{s}'}
      e_{{\ell}_1' \cdots {\ell}_{s}'}) 
    \in (x_1, \ldots, x_{N}) L_{s-1}. 
  \end{equation}
  Put $\sigma = \deg e_{{\ell}_1 \cdots {\ell}_{s}}$. 
  Then $\beta_{s, \sigma} (S/I) \neq 0$. 
\end{proposition}
\begin{proof}
  We prove $[\Tor_s^S (K, S/I)]_{\sigma} \neq 0$. 
  Since the Lyubeznik resolution 
  is a free resolution of $S/I$, 
  it is sufficient to prove that the $s$th homology 
  of $K \otimes L_{\bullet}$ in degree $\sigma$ does not vanish. 
  We set 
  \begin{displaymath}
    \xi = e_{{\ell}_1 \cdots {\ell}_{s}}
       +\sum_{\{ {\ell}_1', \ldots, {\ell}_{s}' \} 
            \neq \{ {\ell}_1, \ldots, {\ell}_s \}} 
        a_{{\ell}_1' \cdots {\ell}_{s}'}
        e_{{\ell}_1' \cdots {\ell}_{s}'}. 
  \end{displaymath}
  Note that we may assume that $\xi$ is a homogeneous elements 
  of degree $\sigma$. 
  Then the assumption (\ref{eq:max}) implies that 
  $1 \otimes \xi \in \Ker (1 \otimes d_s)$. 
  Also $1 \otimes \xi \notin \Image (1 \otimes d_{s+1})$ because of 
  the maximality of $e_{{\ell}_1 \cdots {\ell}_{s}}$. 
  Hence we have $[\Tor_s^S (K, S/I)]_{\sigma} \neq 0$, as required. 
\end{proof}

\par
Let $I'$ and $I''$ be monomial ideals of $S' = K[x_1, \ldots, x_{N'}]$ and 
$S'' = K[y_1, \ldots, y_{N''}]$ respectively. 
Put $S = K[x_1, \ldots, x_{N'}, y_1, \ldots, y_{N''}]$. 
Then we can consider the monomial ideal $I = I'S + I'' S$ of $S$. 
It is easy to see that 
if $\beta_{i', {\sigma}'} (S'/I') \neq 0$ and 
$\beta_{i'', {\sigma}''} (S''/I'') \neq 0$, 
then $\beta_{i'+i'', ({\sigma}', {\sigma}'')} (S/I) \neq 0$. 
Now let us consider the case where $I$ has other monomial generators. 
Precisely, when a monomial ideal $I \subset S$ satisfies 
\begin{equation}
  \label{eq:generators}
  \mathcal{G} (I) \supset \mathcal{G} (I'S) \sqcup \mathcal{G} (I'' S), 
\end{equation}
can we obtain any non-vanishment of Betti numbers of $S/I$ 
from that of $S'/I'$ and $S''/I''$? 

\par
We give a partial answer to the above question 
observing Lyubeznik resolutions again. 
We first fix notations. 
We use indices ${\ell}$, ${\ell}'$, and ${\ell}''$ 
for $I$, $I'$ and $I''$, respectively. 
For simplicity, we denote an $L$-admissible symbol by 
$[{\ell}_1, \ldots, {\ell}_s]$ 
instead of $e_{{\ell}_1 \cdots {\ell}_s}$ and so on. 
We fix orders on $\mathcal{G} (I')$ and $\mathcal{G} (I'')$ 
respectively and consider the Lyubeznik resolutions with respect to the orders. 
Since $I$ satisfies (\ref{eq:generators}), $I$ can be written as 
$I = I'S + I''S + J$, where $J$ is a monomial ideal of $S$ with 
$\mathcal{G} (I) = \mathcal{G} (I'S) \sqcup \mathcal{G} (I''S) 
   \sqcup \mathcal{G} (J)$. 
Then we fix the order on $\mathcal{G} (I)$ as follows: 
first we order $\mathcal{G} (I'S)$ as the same order as $\mathcal{G} (I')$, 
next we order $\mathcal{G} (I''S)$ as the same order as $\mathcal{G} (I'')$, 
and finally we order $\mathcal{G} (J)$ arbitrarily. 
We denote this order simply by 
$\mathcal{G} (I'S), \mathcal{G} (I''S), \mathcal{G} (J)$. 
Let $[{\ell}_1, \ldots, {\ell}_s]$ be an $L$-admissible symbol of $I$ 
with respect to the above order such that 
each ${\ell}_t$ corresponds to a monomial in 
$\mathcal{G} (I') \sqcup \mathcal{G} (I'')$, that is, 
\begin{displaymath}
  [{\ell}_1, \ldots, {\ell}_s]
  = [{\ell}_1', \ldots, {\ell}_{s'}', \# \mathcal{G} (I') + {\ell}_1'', 
      \ldots, \# \mathcal{G} (I') + {\ell}_{s''}'']. 
\end{displaymath}
Then we say that $[{\ell}_1, \ldots, {\ell}_s]$ is the product of 
$[{\ell}_1', \ldots, {\ell}_{s'}']$ and $[{\ell}_1'', \ldots, {\ell}_{s''}'']$ 
and denote it by 
\begin{displaymath}
  [{\ell}_1', \ldots, {\ell}_{s'}'][{\ell}_1'', \ldots, {\ell}_{s''}'']. 
\end{displaymath}

\par
We note that, for any $L$-admissible symbol 
of $I'$ and any $L$-admissible symbol of $I''$, 
the product of these is an $L$-admissible symbol of $I$ with respect to 
the above order. 

\par
Now we answer the question posed above with some maximal assumption. 
In the following proposition, we use similar notations defined above. 
\begin{proposition}
  \label{Lyu-resBettiMany}
  Let $K$ be a field. We consider $r$ polynomial rings 
  $S^{(1)}, \ldots, S^{(r)}$ which have no common variables: 
  $S^{(k)} = K[x_1^{(k)}, \ldots, x_{N_k}^{(k)}]$. 
  Let $I_k$ be a monomial ideal of $S^{(k)}$ with the minimal 
  system of monomial generators $m_1^{(k)}, \ldots, m_{\mu_k}^{(k)}$, 
  and $L_{\bullet}^{(k)}$ be the Lyubeznik resolution of $I_k$ with 
  the above order of generators. 
  Assume that there exist an $L$-admissible symbol 
  ${\tau}^{(k)} \in L_{s_k}^{(k)}$, 
  $L$-admissible symbols 
  ${{\tau}'}^{(k)} \in L_{s_k}^{(k)}$ (${{\tau}'}^{(k)} \neq {\tau}^{(k)}$), 
  and elements $a_{{{\tau}'}^{(k)}}^{(k)} \in S^{(k)}$ such that
  \begin{displaymath}
    d_{s_k}^{(k)} ({\tau}^{(k)} 
     +\sum_{{{\tau}'}^{(k)} \neq {\tau}^{(k)}} 
      a_{{\tau}'^{(k)}}^{(k)}
      {\tau'}^{(k)}) 
    \in (x_1^{(k)}, \ldots, x_{N_k}^{(k)}) L_{s_k-1}^{(k)}. 
  \end{displaymath}
  Put $\sigma_k = \deg {\tau}^{(k)}$. 
  
  \par
  Let $S$ be a polynomial ring over $K$ in variables $x_j^{(k)}$ 
  $(1 \leq k \leq r$; $1 \leq j \leq N_k)$. 
  Let us consider the following monomial ideal of $S$: 
  \begin{displaymath}
    I = I_1 S + \cdots + I_r S + J, 
  \end{displaymath}
  where $J$ is a monomial ideal whose monomial generators do not 
  belong to $S^{(k)}$ for all $1 \leq k \leq r$. 
  Suppose that $L$-admissible symbol 
  $\prod_{k=1}^{r} {\tau}^{(k)}$ 
  is maximal with the following order of $\mathcal{G} (I)$: 
  \begin{displaymath}
    m_1^{(1)}, \ldots, m_{\mu_1}^{(1)}, \  \ldots, \  
    m_1^{(r)}, \ldots, m_{\mu_r}^{(r)}, \  \mathcal{G} (J). 
  \end{displaymath}
  Then 
  \begin{displaymath}
    \beta_{s_1+ \cdots + s_r, (\sigma_1, \ldots, \sigma_r)} 
      (S/I) \neq 0. 
  \end{displaymath}
\end{proposition}
\begin{remark}
  The maximality of the $L$-admissible symbol 
  $\prod_{k=1}^{r} {\tau}^{(k)}$ 
  does not depend on the order of $\mathcal{G} (J)$. 
\end{remark}
\begin{proof}
  We set 
  \begin{displaymath}
    {\xi}_k = {\tau}^{(k)} 
            +\sum_{{{\tau}'}^{(k)} \neq {\tau}^{(k)}} 
              a_{{\tau}'^{(k)}}^{(k)}
              {\tau'}^{(k)} 
  \end{displaymath}
  and 
  \begin{displaymath}
    {\xi} = \prod_{k=1}^r {\xi}_k. 
  \end{displaymath}
  (The meaning of the notation is similar to the one 
  which we defined before the proposition.) 
  Note that we may assume that each ${\xi}_k$ is homogeneous, and thus, 
  $\xi$ is homogeneous. 
  Then the maximal $L$-admissible symbol 
  $\prod_{k=1}^{r} \tau^{(k)}$ appears 
  as a summand of $\xi$ with coefficient $1$. 
  Moreover, 
  \begin{displaymath}
    d_{s_1 + \cdots + s_r} ({\xi}) 
    = \sum_{k=1}^r (-1)^{s_1 + \cdots + s_{k-1}} 
        {\xi}_1 \cdots {\xi}_{k-1} d_{s_k}^{(k)} ({\xi}_k)
        {\xi}_{k+1} \cdots {\xi}_r 
    \subset \mathfrak{m} L_{s_1 + \cdots + s_r - 1}, 
  \end{displaymath}
  where $\mathfrak{m}$ is the unique graded maximal ideal of $S$. 
  Therefore by Proposition \ref{Lyu-resBetti}, we have the desired conclusion. 
\end{proof}

\section{Proof of the main theorem}
\label{sec:proof}
In this section, we prove Theorem \ref{MainIntro}. 
Recall that Theorem \ref{MainIntro} is a generalization 
of \cite[Proposition 2.5]{Katzman} (Theorem \ref{KatzmanBetti}) and 
\cite[Theorem 3.1]{Kimura} (Theorem \ref{KimuraBetti}). 
Katzman proved 
\cite[Proposition 2.5]{Katzman} 
by using of Taylor resolution. 
After, the author proved 
\cite[Theorem 3.1]{Kimura}
from the observation of Lyubeznik resolutions given 
by Barile \cite[Remark 1]{Barile05}; see the previous section. 
On the other hand, Theorem \ref{MainIntro} is proved by 
Propositions \ref{Lyu-resBetti} and \ref{Lyu-resBettiMany}, 
which were obtained by observing Lyubeznik resolutions more detailed. 

\par
First, we prove the following lemma. 
\begin{lemma}
  \label{CompleteBipartite}
  Let $G$ be a graph on a vertex set $V$. 
  Assume that $G$ contains a complete bipartite graph of type $(m,n)$ 
  as a spanning subgraph. 
  Then by considering suitable order on $\mathcal{G} (I(G))$, 
  the edge ideal $I(G)$ admits the assumption of Proposition \ref{Lyu-resBetti} 
  with $s= m+n-1$ and $\sigma = V$. 
  In particular, $\beta_{m+n-1, V} (S/I(G)) \neq 0$. 
\end{lemma}
\begin{proof}
  Since $G$ contains a complete bipartite graph of type $(m,n)$ 
  as a spanning subgraph, 
  the vertex set $V$ is decomposed as 
  $V = \{ u_1, \ldots, u_m \} \sqcup \{ v_1, \ldots, v_n \}$ 
  so that $\{ u_{\alpha}, v_{\beta} \} \in E(G)$ for all $\alpha, \beta$. 
  We order the elements of $\mathcal{G} (I(G))$ as follows: 
  \begin{equation}
    \label{order1}
    \begin{aligned}
      &u_1 v_1, \; u_2 v_1, \; \ldots, \; u_m v_1, \\
      &\quad u_1 v_2, \; u_2 v_2, \; \ldots, \; u_m v_2, \\
      &\qquad \qquad \ldots, \\
      &\qquad \quad u_1 v_n, \; u_2 v_n, \; \ldots, \; u_m v_n, 
       \  \text{other generators}. 
    \end{aligned}
  \end{equation}
  Let $L_{\bullet}$ be the Lyubeznik resolution of $I(G)$ with respect to 
  the order (\ref{order1}). 
  Consider the following subsequence of (\ref{order1}): 
  \begin{equation}
    \label{suborder1}
    \begin{aligned}
      &u_1 v_1, \; u_2 v_1, \; \ldots, \; u_{t_1} v_1, \\
      &\quad u_{t_1} v_2, \; u_{t_1 + 1} v_2, \; \ldots, \; u_{t_2} v_2, \\
      &\qquad \qquad \ldots, \\
      &\qquad \quad u_{t_{n-1}} v_n, \; u_{t_{n-1} + 1} v_n, \; \ldots, \; 
                      u_{m} v_n, 
    \end{aligned}
  \end{equation}
  where $1 \leq t_1 \leq t_2 \leq \cdots \leq t_{n-1} \leq m$. 
  It is easy to see that the corresponding symbol, 
  we denote it by $\tau (t_1, \ldots, t_{n-1})$, 
  is $L$-admissible and belongs to $L_{m+n-1}$. 
  Note that $\deg \tau (t_1, \ldots, t_{n-1}) = V$. 

  \par
  We claim that $\tau (m, \ldots, m)$ is a maximal $L$-admissible symbol. 
  Note that $\tau (m, \ldots, m)$ corresponds to the following subsequence 
  of (\ref{order1}): 
  \begin{equation}
    \label{max-order1}
    u_1 v_1, \; u_2 v_1, \; \ldots, \; u_m v_1, \  
    u_m v_2, \; \ldots, \; u_m v_n. 
  \end{equation}
  Let $e'$ be an edge of $G$ whose corresponding monomial $M'$
  does not lie in (\ref{max-order1}). 
  To prove that $\tau (m, \ldots, m)$ is maximal, 
  it is sufficient to show that 
  the symbol obtained by adding $M'$ to (\ref{max-order1}), we denote it by 
  ${\tau}'$, is not $L$-admissible. 
  Note that $M'$ does not lie in the first line of (\ref{order1}). 
  Suppose that $e'$ contains $u_{\alpha}$. 
  Since the monomial $u_m v_1 \cdot u_{\alpha}$ is divisible by 
  $u_{\alpha} v_1$, 
  ${\tau}'$ is not $L$-admissible unless $u_{\alpha} = u_m$. 
  Therefore we may assume $e' = \{ u_m, v_{\beta} \}$. 
  But $u_m v_{\beta}$ lies in (\ref{max-order1}), a contradiction. 
  Thus $e' \subset \{ v_1, \ldots, v_n \}$. Then $e'$ contains $v_{\beta}$ 
  for some $\beta < n$ and $M'$ belongs to other generators of (\ref{order1}). 
  Since $u_m v_n \cdot v_{\beta}$ is divisible by $u_m v_{\beta}$, 
  ${\tau}'$ is not $L$-admissible, as required. 

  \par
  Now we consider 
  \begin{displaymath}
    \xi := \sum_{1 \leq t_1 \leq \cdots \leq t_{n-1} \leq m} 
           (-1)^{t_1 + \cdots + t_{n-1}} \tau (t_1, \ldots, t_{n-1}). 
  \end{displaymath}
  Let $\mathfrak{m}$ be the graded maximal ideal of $K[V]$. 
  Then 
  \begin{displaymath}
    \begin{aligned}
    &d_{m+n-1} (\tau (t_1, \ldots, t_{n-1})) \\
    &\equiv \sum_{\genfrac{}{}{0pt}{}{1 < \alpha < n}{t_{\alpha - 1} < t_{\alpha}}} 
      ((-1)^{t_{\alpha - 1} + \alpha - 2} \tau (t_1, \ldots, t_{n-1}; {\alpha}-) 
       + (-1)^{t_{\alpha} + \alpha - 2} \tau (t_1, \ldots, t_{n-1}; {\alpha}+)) \\
    &\quad + (-1)^{t_1-1} {\tau}_1 + (-1)^{t_{n-1} + n-2} {\tau}_{n}
     \mod \mathfrak{m} L_{m+n-2}, 
    \end{aligned}
  \end{displaymath}
  where 
  \begin{displaymath}
    \begin{aligned}
    {\tau}_1 &= \left\{
    \begin{alignedat}{3}
      &0, &\quad &\text{if $t_1 = 1$}, \\
      &\tau (t_1, \ldots, t_{n-1}; 1+), &\quad &\text{if $t_1 > 1$}, 
    \end{alignedat}
    \right. \\
    {\tau}_{n} &= \left\{
    \begin{alignedat}{3}
      &0, &\quad &\text{if $t_{n-1} = m$}, \\
      &\tau (t_1, \ldots, t_{n-1}; n-), &\quad &\text{if $t_{n-1} < m$}, 
    \end{alignedat}
    \right. 
    \end{aligned}
  \end{displaymath}
  and where $\tau (t_1, \ldots, t_{n-1}; \alpha -)$ 
  (resp.\  $\tau (t_1, \ldots, t_{n-1}; \alpha +)$) is an $L$-admissible symbol 
  obtained by omitting $u_{t_{\alpha - 1}} v_{\alpha}$ 
  (resp.\  $u_{t_{\alpha}} v_{\alpha}$) 
  from $\tau (t_1, \ldots, t_{n-1})$.  
  Since 
  \begin{displaymath}
    \begin{aligned}
      \tau (t_1, \ldots, t_{n-1}; \alpha +) 
      &= \tau (t_1, \ldots, t_{\alpha - 1}, t_{\alpha} - 1, 
               t_{\alpha + 1}, \ldots, t_{n-1}; ({\alpha}+1)-), \\
      \tau (t_1, \ldots, t_{n-1}; \alpha -) 
      &= \tau (t_1, \ldots, t_{\alpha-2}, t_{\alpha-1} + 1, 
               t_{\alpha}, \ldots, t_{n-1}; ({\alpha}-1)+) 
    \end{aligned}
  \end{displaymath}
  when $t_{\alpha -1} < t_{\alpha}$, 
  we can easily check that  $d_{m+n-1} (\xi) \in \mathfrak{m} L_{m+n-2}$. 

  \par
  Therefore the assertion follows. 
\end{proof}

Now we prove Theorem \ref{MainIntro}. 
\begin{proof}[Proof of Theorem \ref{MainIntro}]
  By Hochster's formula for Betti numbers (\cite[Theorem 5.5.1]{BrunsHerzog}), 
  we have 
  \begin{displaymath}
    \beta_{i, \sigma} (S/I(G)) = \beta_{i, \sigma} (K[{\sigma}]/I(G_{\sigma})). 
  \end{displaymath}
  Therefore we may assume $V = \sigma$. 
  For each $k = 1, \ldots, r$, 
  let $(m_k, n_k)$ be the type of the complete bipartite subgraph $B_k$. 
  Put 
  $V_k := V(B_k) = \{ u_1^{(k)}, \ldots, u_{m_k}^{(k)} \} 
                   \sqcup \{ v_1^{(k)}, \ldots, v_{n_k}^{(k)} \}$. 
  We may assume that $e_1, \ldots, e_r$ are pairwise $3$-disjoint 
  where $e_k = \{ u_{m_k}^{(k)}, v_{n_k}^{(k)} \}$. 
  We order the elements of $\mathcal{G} (I(G_{V_k}))$ 
  as in the proof of Lemma \ref{CompleteBipartite}. 
  We use the same notation as in Lemma \ref{CompleteBipartite} 
  with upper subscript $(k)$. 
  We order the elements of $\mathcal{G} (I(G))$ as follows: 
  \begin{equation}
    \label{orderM}
    \mathcal{G} (I(G_{V_1})), \; \mathcal{G} (I(G_{V_2})), \; \ldots, \; 
    \mathcal{G} (I(G_{V_r})), \; \text{other generators}. 
  \end{equation}
  We set $\tau = \prod_{k=1}^r \tau (m_k, \ldots, m_k)^{(k)}$, 
  which is an $L$-admissible symbol of $I(G)$. 
  Since $\tau (m_k, \ldots, m_k)^{(k)} \in L_{m_k + n_k -1}^{(k)}$ 
  and $\deg \tau (m_k, \ldots, m_k)^{(k)} = V_k$, 
  we have $\tau \in L_{\# V - r}$ and $\deg \tau = V = \sigma$. 
  Therefore by Propositions \ref{Lyu-resBetti} and \ref{Lyu-resBettiMany} 
  and Lemma \ref{CompleteBipartite}, 
  for proving the theorem, 
  it is enough to show that $\tau$ is a maximal $L$-admissible symbol, 
  i.e., if we add the generator of $I(G)$ corresponding to 
  another edge $e'$ to $\tau$, then it is not an 
  $L$-admissible symbol. 

  \par
  Let $M'$ be the corresponding monomial to $e'$, 
  and ${\tau}'$ a symbol 
  obtained by adding $M'$ to $\tau$. 
  When $e' \in E(G_{V_k})$, the maximality of $\tau (m_k, \ldots, m_k)^{(k)}$ 
  shows that ${\tau}'$ is not $L$-admissible. 
  Thus we may assume that $M'$ lies in 
  other generators of (\ref{orderM}). 
  Then, as we saw at the proof of Lemma \ref{CompleteBipartite}, 
  $e'$ must be consists of two vertices in 
  $\bigcup_{k=1}^r \{ u_{m_k}^{(k)}, v_{n_k}^{(k)} \}$. 
  Since $e' \neq e_1, \ldots, e_r$, this contradicts to 
  the pairwise $3$-disjointness of $e_1, \ldots, e_r$. 
\end{proof}

\section{Graded Betti numbers in the linear strand}
\label{sec:LinearStrand}
The converse of Theorem \ref{MainIntro} does not hold since
$\reg (S/I(G)) > a(G)$ holds in general 
as noted in Section \ref{sec:pre}. 
However in this section, we show that the converse of Theorem \ref{MainIntro} 
is true for the graded Betti numbers in the linear strand; 
see Proposition \ref{LinearStrand}. 
On there we will see that complete bipartite subgraphs naturally arise. 
The graded Betti numbers in the linear strand was investigated by 
Roth and Van Tuyl \cite{RVT}. They gave some formula for such 
Betti numbers in terms of graphic invariants. 
However we obtain another characterization of these Betti numbers. 
Also we review the graded Betti numbers of 
the edge ideals which have linear resolutions. 
In particular we give a combinatorial description for 
the projective dimension of such ideals. 

\par
We first recall the definition of the Stanley--Reisner ring. 
A \textit{simplicial complex} $\Delta$ on the vertex set $V$ is a collection 
of subsets of $V$ with the properties: 
(i) $\{ v \} \in \Delta$ for all $v \in V$; 
(ii) If $F \in \Delta$, then $G \in \Delta$ for all $G \subset F$. 
We associate a simplicial complex $\Delta$ on $V$ with 
the squarefree monomial ideal $I_{\Delta}$ of $S=K[V]$ by
\begin{displaymath}
  I_{\Delta} = ( x_F \; : \; F \subset V,\  F \notin \Delta), 
  \qquad x_F = \prod_{v \in F} v, 
\end{displaymath}
which is called the \textit{Stanley--Reisner ideal} of $\Delta$. 
The quotient ring $K[{\Delta}] := S/I_{\Delta}$ is called the 
\textit{Stanley--Reisner ring} of $\Delta$. 

\par
Let $G$ be a finite simple graph on the vertex set $V$. 
A subset $W \subset V$ is called an \textit{independent set} if 
any two vertices $w_1, w_2 \in W$ are not adjacent in $G$. 
The collection of all independent sets of $G$ forms a simplicial complex 
on $V$; we denote it by $\Delta (G)$. 
By Hochster's formula \cite[Theorem 5.5.1]{BrunsHerzog}, 
we have for $\sigma \subset V$ 
\begin{displaymath}
  \beta_{i, \sigma} (S/I(G)) 
  = \dim_K \tilde{H}_{|{\sigma}| - i - 1} ((\Delta (G))_{\sigma}; K), 
\end{displaymath}
where $(\Delta (G))_{\sigma}$ denotes the restriction of $\Delta (G)$ 
on $\sigma$, 
i.e., the simplicial complex consists of all faces of $\Delta (G)$ 
contained in $\sigma$. Note that $(\Delta (G))_{\sigma} = \Delta (G_{\sigma})$. 
Let us consider $\beta_{i, \sigma} (S/I(G))$, where $|{\sigma}| =i+1$, 
which corresponds to a graded Betti number in the linear strand of $I(G)$. 
Then, 
\begin{equation}
  \label{betti0}
  \begin{aligned}
    \beta_{i, \sigma} (S/I(G)) &= \dim_K \tilde{H}_0 (\Delta (G_{\sigma}); K) \\
      &= (\text{number of connected components of $\Delta (G_{\sigma})$}) - 1. 
  \end{aligned}
\end{equation}
Roth and Van Tuyl \cite{RVT} studied these graded Betti numbers 
on $\mathbb{N}$-graded case. 
In particular they gave an exact formulas for some classes of graphs 
in terms of graphic invariants. 
However from the view point of Theorem \ref{MainIntro}, 
we give another characterization of these Betti numbers. 
We define $c(G)$ as the maximum number $c$ so that 
$G$ contains a complete $c$-partite graph as a spanning subgraph; 
if $G$ does not contain a complete bipartite graph as a spanning subgraph, 
then we set $c(G) = 1$. 
\begin{proposition}
  \label{LinearStrand}
  Let $G$ be a finite simple graph on the vertex set $V$. 
  Let $i \geq 1$ be an integer and 
  $\sigma \subset V$ with $|{\sigma}| = i+1$. 
  Then
  $\beta_{i, \sigma} (S/I(G)) = c(G_{\sigma}) - 1$. 

  \par
  In particular, with the same assertions for $i$ and $\sigma$ as above, 
  $\beta_{i, \sigma} (S/I(G)) \neq 0$ if and only if 
  $G_{\sigma}$ contains a complete bipartite graph as a spanning subgraph. 
\end{proposition}
\begin{proof}
  By (\ref{betti0}), 
  $\beta_{i, \sigma} (S/I(G)) \neq 0$ if and only if 
  $\Delta (G_{\sigma})$ is disconnected. 
  When this is the case, 
  the vertex set $\sigma$ of $\Delta (G_{\sigma})$ has a bipartition 
  $\sigma = {\sigma}_1 \sqcup {\sigma}_2$ such that 
  $\{ v_1, v_2 \} \notin \Delta (G_{\sigma})$ 
  for all $v_1 \in {\sigma}_1$, $v_2 \in {\sigma}_2$. 
  Since $\{ v_1, v_2 \} \notin \Delta (G_{\sigma})$ is equivalent to 
  $\{ v_1, v_2 \} \in E(G_{\sigma})$, this occurs if and only if 
  $G_{\sigma}$ contains complete bipartite graph 
  as a spanning subgraph 
  (with the same bipartition $\sigma = {\sigma}_1 \sqcup {\sigma}_2$). 

  \par
  Now assume that $\beta_{i, \sigma} (S/I(G)) \neq 0$ 
  and set $c := c(G_{\sigma})$. 
  Then $\sigma$ can be partitioned as 
  $\sigma_1 \sqcup \cdots \sqcup \sigma_{c}$ 
  with the property that $\{ v_s, v_t \} \in E(G_{\sigma})$ whenever 
  $v_s \in \sigma_s$ and $v_t \in \sigma_t$ where $s \neq t$. 
  This implies that $\Delta (G_{\sigma})$ has at least $c$ 
  connected components. 
  On the other hand, it also follows that if $\Delta (G_{\sigma})$ has $c'$ 
  connected components, then $G_{\sigma}$ contains a 
  complete $c'$-partite graph as a spanning subgraph. 
  As a consequence we have $c' = c$. Then the assertion follows 
  from (\ref{betti0}). 
\end{proof}

\par
For a finite simple graph $G$ on the vertex set $V$, 
we denote by $G^c$ the \textit{complementary graph}, that is a graph on $V$ 
where two vertices $v_1, v_2 \in V$ are adjacent in $G^c$ if and only if 
those are not adjacent in $G$. 
A graph $G$ is called \textit{co-chordal} when $G^c$ is chordal. 
Fr\"{o}berg \cite[Theorem 1]{Froberg} proved that $I(G)$ 
has a linear resolution if and only if $G$ is a co-chordal graph. 
When this is the case, 
$\beta_{i, \sigma} (S/I(G)) = 0$ for all $i \geq 1$ and for 
all $\sigma \subset V$ with $|{\sigma}| \geq i+2$. 
Thus we have the following corollary. 
\begin{corollary}
  \label{co-chordalBetti}
  Let $G$ be a co-chordal graph. 
  Then for $i \geq 1$, $\beta_{i, \sigma} (S/I(G)) \neq 0$ if and only if 
  $|{\sigma}| = i + 1$ and $G_{\sigma}$ contains a complete bipartite subgraph 
  as a spanning subgraph. 
\end{corollary}
This corollary implies that
the condition in Theorem \ref{MainIntro} gives a 
necessary and sufficient one for non-vanishing of graded 
Betti numbers of edge ideals of co-chordal graphs. 

\par
Since Corollary \ref{co-chordalBetti} is a characterization 
of non-vanishing of Betti numbers of 
the edge ideal of a co-chordal graph, 
we can characterize the regularity 
and the projective dimension of such an edge ideal, 
though the result on the regularity is trivial. 
\begin{corollary}
  \label{co-chordalProj}
  Let $G$ be a co-chordal graph. 
  Then $\reg (S/I(G)) = a(G) = 1$ and 
  \begin{displaymath}
      \pd (S/I(G)) = \max \{ m+n-1 \; : \; 
        \text{$G$ contains complete bipartite subgraph of type $(m,n)$} \}. 
  \end{displaymath}
\end{corollary}
\begin{remark}
  The graded Betti numbers of co-chordal graph had been studied 
  by Dochtermann and Engstr\"{o}m \cite[Section 3.1]{DE}. 
  The statement about the projective dimension 
  of Corollary \ref{co-chordalProj} 
  is a rephrasing of \cite[Corollary 3.3]{DE}. 

  \par
  Also Chen \cite{Chen} constructed a minimal free resolution for 
  the edge ideal of a co-chordal graph (see also \cite{Horwitz}) and 
  gave formulas for ($\mathbb{N}$)-graded Betti numbers 
  and the projective dimension of such an ideal (\cite[Corollary 5.2]{Chen}), 
  which are different from ours. 
\end{remark}

Eliahou and Villarreal \cite{EV} conjectured that 
when $G$ is a co-chordal graph, 
the projective dimension $\pd (S/I(G))$ is equal to the maximum degree of 
vertices of $G$ (see \cite[Conjecture 4.13]{GV05}). 
Although this conjecture is true for some classes of graphs 
(e.g., \cite[Corollary 4.12]{GV05}, \cite[Theorem 2.13]{MK}), 
in general, it is not true; see \cite{MK}. 
Actually, considering the maximum degree of vertices of $G$ 
is equal to considering the maximum cardinality of the vertex set of bouquets 
those are subgraph of $G$. 
Corollary \ref{co-chordalProj} means that we need to consider 
arbitrary complete bipartite subgraphs 
though a bouquet is a complete bipartite graphs of type $(1,n)$. 

\par
The following examples are graphs 
whose edge ideals have $2$-linear resolutions 
for which we can adapt Corollary \ref{co-chordalProj} 
(the $4$-cycle $C_4$ is noted in \cite{MK}). 
These are also of examples showing that 
we should apply 
Theorem \ref{MainIntro} instead of \cite[Theorem 3.1]{Kimura} \
to know non-vanishing of 
Betti numbers of the edge ideals. 
\begin{example}
  \label{C4}
  Let us consider $4$-cycle $C_4$: 
  
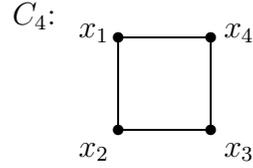
\begin{center}
    \begin{picture}(120,60)(-50,0)
      \put(-30,50){$C_4$:}
      \put(10,10){\circle*{4}}
      \put(45,45){\circle*{4}}
      \put(10,45){\circle*{4}}
      \put(45,10){\circle*{4}}
      \put(-5,0){$x_2$}
      \put(50,45){$x_4$}
      \put(-5,45){$x_1$}
      \put(50,0){$x_3$}
      \put(10,10){\line(0,1){35}}
      \put(10,10){\line(1,0){35}}
      \put(45,45){\line(0,-1){35}}
      \put(45,45){\line(-1,0){35}}
    \end{picture}
  \end{center}
  Then the ($\mathbb{N}$-graded) Betti diagram of $S/I(C_4)$ is 
  \begin{center}
  \begin{tabular}{c|rrrr}
    $j \backslash i$ & $0$ & $1$ & $2$ & $3$ \\ \hline
    $0$ & $1$ & & & \\
    $1$ & & $4$ & $4$ & $1$
  \end{tabular}
  \end{center}
  where the entry of the $j$th row and $i$th column stands for 
  $\beta_{i,i+j} (S/I(C_4))$. 

  \par
  We focus on $\beta_{3,4} (S/I(C_4))$. 
  Since $C_4$ does not have a complete bipartite subgraph of type $(1,3)$ 
  (bouquet), 
  we cannot obtain 
  $\beta_{3,4} (S/I(C_4)) \neq 0$ from \cite[Theorem 3.1]{Kimura}. 
  On the other hand, $C_4$ has a complete bipartite subgraph 
  of type $(2,2)$, 
  we have $\beta_{3,4} (S/I(C_4)) \neq 0$ by Theorem \ref{MainIntro}. 
  Also, $\pd (S/I(C_4)) = 3 = 2+2-1$. 
\end{example}

\begin{example}
  \label{FerrersGraph}
  Next we consider the following graph $G$: 
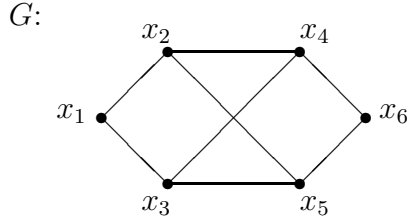
  \begin{center}
  \begin{picture}(170,80)(-20,0)
    \put(-10,70){$G$:}
    \put(50,60){\circle*{4}}
    \put(100,60){\circle*{4}}
    \put(50,10){\circle*{4}}
    \put(100,10){\circle*{4}}
    \put(25,35){\circle*{4}}
    \put(125,35){\circle*{4}}
    \put(40,65){$x_2$}
    \put(100,65){$x_4$}
    \put(40,0){$x_3$}
    \put(100,0){$x_5$}
    \put(8,35){$x_1$}
    \put(130,35){$x_6$}
    \put(25,35){\line(1,1){25}}
    \put(25,35){\line(1,-1){25}}
    \put(125,35){\line(-1,1){25}}
    \put(125,35){\line(-1,-1){25}}
    \put(50,60){\line(1,0){50}}
    \put(50,60){\line(1,-1){50}}
    \put(50,10){\line(1,0){50}}
    \put(50,10){\line(1,1){50}}
  \end{picture}
  \end{center}
  Then the ($\mathbb{N}$-graded) Betti diagram of $S/I(G)$ is 
  \begin{center}
  \begin{tabular}{c|rrrrr}
    $j \backslash i$ & $0$ & $1$ & $2$ & $3$ & $4$ \\ \hline
    $0$ & $1$ & & & & \\
    $1$ & & $8$ & $14$ & $9$ & $2$ 
  \end{tabular}
  \end{center}
  We focus on $\beta_{4,5} (S/I(G))$. 
  Since $G$ does not have a complete bipartite subgraph of type $(1,4)$ 
  (bouquet), we cannot obtain 
  $\beta_{4,5} (S/I(G)) \neq 0$ from \cite[Theorem 3.1]{Kimura}. 
  On the other hand, $G$ has a complete bipartite subgraph of type $(2,3)$, 
  we have $\beta_{4,5} (S/I(G)) \neq 0$ by Theorem \ref{MainIntro}. 

  \par
  Actually the induced subgraphs of $G$ on both 
  $\{ x_1, x_2, x_3, x_4, x_5 \}$ and 
  $\{ x_2, x_3, x_4, x_5, x_6 \}$ are complete bipartite graph of type $(2,3)$. 

  \par
  Also, $\pd (S/I(G)) = 4 = 2+3-1$. 
\end{example}

The graphs in Examples \ref{C4} and \ref{FerrersGraph} 
are Ferrers graphs with shape $(2,2)$ and $(3,3,2)$, respectively. 
Ferrers graphs lie in the class of co-chordal graphs. 
Corso and Nagel \cite[Theorem 2.1]{CN} gave an exact formula for 
the ${\mathbb{N}}$-graded Betti numbers of edge ideals 
of Ferrers graphs. 
Later Dochtermann and Engstr\"{o}m \cite{DE} 
provided a combinatorial meaning of these Betti numbers, 
which fits our point of view. 
\begin{theorem}[{\cite[Theorem 3.8]{DE}}]
  Let $i \geq 1$ be an integer. 
  Then the graded Betti number $\beta_{i, \sigma} (S/I(G_{\lambda}))$ 
  of a Ferrers graph $G_{\lambda}$ is $0$ unless $|{\sigma}| = i+1$ 
  and $(G_{\lambda})_{\sigma}$ is a complete bipartite graph, 
  in which case $\beta_{i, \sigma} (S/I(G_{\lambda})) = 1$.  
  In particular, $\beta_{i, i+1} (S/I(G_{\lambda}))$ is equal to 
  the number of induced complete bipartite 
  subgraphs of $G_{\lambda}$ which consist of $i+1$ vertices. 
\end{theorem}
In particular, Ferrers graphs $G_{\lambda}$ are co-chordal graphs with 
$c((G_{\lambda})_{\sigma})$ is either $1$ or $2$ 
for any subsets $\sigma \subset V$.

\section{Cohen--Macaulay bipartite graphs}
\label{sec:CMbipartite}
We continue the investigation of the converse of Theorem \ref{MainIntro}. 
Next, we focus on the class of graphs $G$ satisfying 
$\reg (S/I(G)) = a(G)$. 
Does the converse of Theorem \ref{MainIntro} hold for these graphs? 
In this section, we discuss this problem. 

\par
Note that chordal graphs including forests are such graphs; 
see Theorem \ref{reg=a(G)}. 
On $\mathbb{N}$-grading, the author \cite{Kimura} proved that
the converse of Theorem \ref{MainIntro} (precisely, \cite[Theorem 3.1]{Kimura}) 
is true for these graphs (\cite[Theorem 4.1]{Kimura}). 
(Actually this is true for ${\mathbb{N}}^N$-grading.) 
In this section we focus on Cohen--Macaulay bipartite graphs $G$, 
which also satisfy $\reg (S/I(G)) = a(G)$; see Theorem \ref{reg=a(G)}. 
As a result, we will characterize the projective dimension of Cohen--Macaulay 
bipartite graphs (Corollary \ref{CMbipartiteProj}). 

\par
We first recall the notion of Alexander duality. 
Let $S = K[x_1, \ldots, x_N]$ be a polynomial ring over a field $K$ with 
${\mathbb{N}}^N$-grading. 
Let $I$ be a squarefree monomial ideal of $S$ 
with the minimal prime decomposition 
\begin{displaymath}
  I = P_1 \cap \cdots \cap P_q, 
  \qquad P_{\ell} = (x_{i_{\ell 1}}, \ldots, x_{i_{\ell k_{\ell}}}). 
\end{displaymath}
Then the Alexander dual ideal of $I$, denoted by $I^{\ast}$ 
is the squarefree monomial ideal generated by 
\begin{displaymath}
  x_{i_{\ell 1}} \cdots x_{i_{\ell k_{\ell}}}, \qquad \ell = 1, \ldots, q. 
\end{displaymath}

\par
Bayer, Charalambous, and Popescu \cite{BCP} investigated relations 
on graded Betti numbers of a squarefree monomial ideal 
and its Alexander dual ideal. 
Let $\prec$ denote the partial order on ${\mathbb{N}}^N$: 
for $\sigma = (\sigma_{\ell}), \tau = (\tau_{\ell}) \in {\mathbb{N}}^N$ 
with $\sigma \neq \tau$, 
$\sigma \prec \tau$ if and only if 
${\sigma}_{\ell} \leq {\tau}_{\ell}$ for all $\ell = 1, \ldots, N$. 
We say that $\beta_{i, \sigma} (I)$ is \textit{extremal} if 
$\beta_{j, \tau} (I) = 0$ for all $j \geq i$ and 
$\tau \succ \sigma$ with $|\tau| -|\sigma| \geq j-i$. 
\begin{theorem}[{\cite[Theorem 2.8]{BCP}}]
  \label{extremalAdual}
  Let $I \subset S$ be a squarefree monomial ideal. 
  If $\beta_{r, \sigma} (I^{\ast})$ is extremal, 
  then 
  \begin{displaymath}
    \beta_{r, \sigma} (I^{\ast}) = \beta_{|\sigma| - r, \sigma} (S/I). 
  \end{displaymath}
\end{theorem}
Also, Eagon and Reiner \cite{ER} proved the following theorem. 
\begin{theorem}[{\cite{ER}}]
  \label{Adual}
  Let $I$ be a squarefree monomial ideal. Then   
  \begin{displaymath}
    \reg (I^{\ast}) = \pd (S/I), \qquad
    \pd (I^{\ast}) = \reg (S/I). 
  \end{displaymath}
\end{theorem}
\begin{remark}
  \label{extremal}
  We can obtain $\reg (I^{\ast})$ and $\pd (I^{\ast})$ by 
  only seeing non-zero extremal Betti numbers $\beta_{r, \sigma} (I^{\ast})$. 
\end{remark}

\par
Let $G$ be a finite simple graph on $V$. 
By Theorem \ref{extremalAdual}, 
if $\beta_{r, \sigma} (I(G)^{\ast}) \neq 0$ is extremal, 
then we have $\beta_{|\sigma| - r, \sigma} (S/I(G)) \neq 0$. 
Thus we consider, in this situation, 
whether there exists a set of complete bipartite subgraphs of $G$ 
which guarantees the non-vanishing of 
$\beta_{|\sigma| - r, \sigma} (S/I(G))$ 
via Theorem \ref{MainIntro}. 
Note that when this is positive, we can characterize 
the regularity and the projective dimension in terms of Theorem \ref{MainIntro} 
by Theorem \ref{Adual} and Remark \ref{extremal}. 

\par
In general, we do not know the answer. 
But when $G$ is a Cohen--Macaulay bipartite graph, 
the answer is yes. 
Before stating the result, we give one definition. 

\par
Let $G$ be a finite simple graph. 
Let $\mathcal{B} = \{ B_1, \ldots, B_r \}$ 
be a set of complete bipartite subgraphs of $G$. 
We set 
\begin{displaymath}
  V({\mathcal{B}}) := V(B_1) \cup \cdots \cup V(B_r). 
\end{displaymath}
We say $\mathcal{B}$ is 
a \textit{pairwise $3$-disjoint set of complete bipartite subgraphs} of $G$ 
if $\mathcal{B}$ satisfies the condition (1) and (2) 
of Theorem \ref{MainIntro}. 
\begin{proposition}
  \label{extrem-CMbipartite}
  Let $G$ be a Cohen--Macaulay bipartite graph. 
  If $\beta_{r, \sigma} (I(G)^{\ast}) \neq 0$ is extremal, 
  then there exists a pairwise $3$-disjoint set 
  $\mathcal{B} = \{ B_1, \ldots, B_r\}$ of 
  complete bipartite subgraphs of $G$ with $\# V({\mathcal{B}}) = \sigma$. 
\end{proposition}
In particular, we can recover the following result for the regularity. 
\begin{corollary}[{\cite[Theorem 1.1]{Kummini}; see also \cite[Corollary 3.10]{FHVT}}]
  \label{regualrityCMbipartite}
  Let $G$ be a Cohen--Macaulay bipartite graph. 
  Then $\reg (S/I(G)) = a(G)$. 
\end{corollary}
Also we can characterize the projective dimension 
of a Cohen--Macaulay bipartite graph. 

\begin{corollary}
  \label{CMbipartiteProj}
  Let $G$ be a Cohen--Macaulay bipartite graph. 
  Then 
  \begin{displaymath}
    \pd (S/I(G)) = \max \left\{ \# V ({\mathcal{B}}) - r \; : \; 
    \begin{aligned}
      &\text{$\mathcal{B} = \{ B_1, \ldots, B_r \}$ 
            is a pairwise $3$-disjoint set} \\ 
      &\text{of complete bipartite subgraphs of $G$} 
    \end{aligned}
    \right\}. 
  \end{displaymath}
\end{corollary}
\begin{remark}
  Dao and Schweig \cite[Corollary 5.6 and Remark 6.7]{DS} gave
  a formula for the projective dimension of the edge ideals of 
  sequentially Cohen--Macaulay graphs, which  
  include Cohen--Macaulay bipartite graphs. 
  But their formula is different from our characterization. 
\end{remark}

\par
To prove Proposition \ref{extrem-CMbipartite}, 
we first recall the results by Herzog and Hibi \cite{HH}, 
the structure of a Cohen--Macaulay bipartite 
graph and its graded minimal free resolution. 

\begin{theorem}[{Herzog and Hibi \cite{HH}}]
  \label{CMbipartite}
  Let $G$ be a bipartite graph on the vertex set 
  $V$ with a bipartition 
  $V = \{ x_1, \ldots, x_m \} \cup \{ y_1, \ldots, y_n \}$ 
  without isolated vertices. 
  Then $G$ is Cohen--Macaulay if and only if $n=m$ and
  there exists a relabeling of the variables $x_1, \ldots, x_n$ 
  and $y_1, \ldots, y_n$ satisfying the following $3$ conditions: 
  \begin{enumerate}
  \item[(CM1)] $\{ x_i, y_i \} \in E(G)$ for all $i = 1, \ldots, n$. 
  \item[(CM2)] If $\{ x_i, y_j \} \in E(G)$, then $i \leq j$. 
  \item[(CM3)] If $\{ x_i, y_j \}, \{ x_j, y_k \} \in E(G)$ 
    for distinct $i,j,k$, 
    then $\{ x_i, y_k \} \in E(G)$. 
  \end{enumerate}
\end{theorem}

\par
Let $G$ be a Cohen--Macaulay bipartite graph on 
$V = \{ x_1, \ldots, x_n \} \cup \{ y_1, \ldots, y_n \}$ 
satisfying (CM1), (CM2), and (CM3). 
Then we can associate the poset $P_G$ with $G$: 
the set of elements of $P_G$ is 
$\{ p_1, \ldots, p_n \}$ and $p_i \leq p_j$ if and only if 
$\{ x_i, y_j \} \in E(G)$. 
A poset ideal $\mathcal{I}$ of $P_G$ is a subset of $P_G$ 
with the property: 
if $p_{\ell} \in \mathcal{I}$, then $p_{{\ell}'} \in \mathcal{I}$ 
for all $p_{{\ell}'} \leq p_{\ell}$. 
By definition, a poset ideal is determined by the maximal elements of it. 
We denote the set of maximal elements of $\mathcal{I}$ by $M(\mathcal{I})$. 
The set of all poset ideals of $P_G$ forms a distributive lattice 
$\mathcal{L}_G$. 
We consider the squarefree monomial ideal $H_G := H_{P_G}$ 
generated by $u_{\mathcal{I}}$, $\mathcal{I} \in \mathcal{L}_G$, where
\begin{displaymath}
  u_{\mathcal{I}} = \prod_{p \in \mathcal{I}} x_p 
                    \prod_{p \in P_G \setminus \mathcal{I}} y_p. 
\end{displaymath}
The ideal $H_G$ is in fact the Alexander dual ideal of $I(G)$. 

\par
Herzog and Hibi \cite{HH} gave an explicit minimal $\mathbb{N}$-graded 
free resolution of $H_G$, 
which is also a minimal ${\mathbb{N}}^{2n}$-graded 
free resolution. The $i$th free bases are 
\begin{displaymath}
  e(\mathcal{I},T), 
\end{displaymath}
where $\mathcal{I} \in \mathcal{L}_G$ and $T \subset P_G$ 
such that 
\begin{displaymath}  
  \mathcal{I} \cap T \subset M(\mathcal{I}), \  
  \# (\mathcal{I} \cap T) = i, \  
  \mathcal{I} \cup T = \{ p_1, \ldots, p_n \}. 
\end{displaymath}
The degree of $e(\mathcal{I},T)$ is
\begin{displaymath}
  \deg e(\mathcal{I},T) = \deg U_{\mathcal{I}}, 
  \qquad U_{\mathcal{I}} 
           := u_{\mathcal{I}} \prod_{p \in \mathcal{I} \cap T} y_p. 
\end{displaymath}
A free basis $e(\mathcal{I}, T)$ corresponds to a non-zero graded Betti number. 
Moreover it is characterized in terms of $\mathcal{L}_G$. 
\begin{lemma}[{Herzog and Hibi \cite[Corollary 2.2]{HH}}]
  \label{Boolean}
  We use the same notation as above. 
  The correspondence 
  \begin{displaymath}
    e(\mathcal{I}, T) \mapsto 
    [\mathcal{I} \setminus (\mathcal{I} \cap T), \mathcal{I}]
  \end{displaymath}
  gives a bijection between the set of basis elements 
  $e(\mathcal{I}, T)$ and the set of Boolean sublattices of $\mathcal{L}_G$. 

  \par
  In particular, 
  $e(\mathcal{I}, T)$ corresponds to a non-zero extremal Betti number 
  if and only if $[\mathcal{I} \setminus (\mathcal{I} \cap T), \mathcal{I}]$ 
  is a maximal Boolean sublattice of $\mathcal{L}_G$. 
\end{lemma}
\begin{remark}
  The statement about extremal Betti numbers is implicit in \cite{HH}; 
  see also Mohammadi and Moradi \cite[Corollary 1.7]{MM}. 
\end{remark}

\par
Now we prove Proposition \ref{extrem-CMbipartite}
\begin{proof}[Proof of Proposition \ref{extrem-CMbipartite}]
  Let $G$ be a Cohen--Macaulay bipartite graph on the vertex set 
  $V = \{ x_1, \ldots, x_n \} \cup \{ y_1, \ldots, y_n \}$ with 
  the condition (CM1), (CM2), and (CM3). 
  Let $\beta_{r, \sigma} (H_G) \neq 0$ be an extremal Betti number. 
  Then $\beta_{|\sigma| - r, \sigma} (S/I) \neq 0$ and we will construct 
  a set of complete bipartite subgraphs $\{ B_1, \ldots, B_r \}$ satisfying 
  the condition (1) and (2) of Theorem \ref{MainIntro}. 

  \par
  Since $\beta_{r, \sigma} (H_G) \neq 0$ is an extremal Betti number, 
  there is the free basis $e(\mathcal{I}, T)$ 
  with $\# (\mathcal{I} \cap T) = r$ 
  and $\deg e(\mathcal{I},T) = \sigma$. Moreover, by extremality, 
  $[\mathcal{I} \setminus (\mathcal{I} \cap T), \mathcal{I}]$ is a 
  maximal Boolean sublattice of $\mathcal{L}_G$. 
  The maximality of 
  $[\mathcal{I} \setminus (\mathcal{I} \cap T), \mathcal{I}]$ implies 
  $\mathcal{I} \cap T = M(\mathcal{I})$. 
  Set $M(\mathcal{I}) = \{ p_{{\ell}_1}, \ldots, p_{{\ell}_r} \}$. 
  Since $M(\mathcal{I})$ is the set of maximal elements of $\mathcal{I}$, 
  it forms an antichain. In terms of the graph $G$, the edges 
  $\{ x_{{\ell}_1}, y_{{\ell}_1} \}, \ldots, \{ x_{{\ell}_r}, y_{{\ell}_r} \}$ 
  are pairwise $3$-disjoint in $G$. 
  We define $V_1$ to be the subset of $V$ 
  consisting of all vertex $z \in V$ which divides $U_{\mathcal{I}}$ and 
  one of $\{ z, x_{{\ell}_1} \}$, $\{ z, y_{{\ell}_1} \}$ is an edge of $G$. 
  Next we define $V_2$ to be the subset of $V \setminus V_1$  
  consisting of all vertex $z \in V \setminus V_1$ 
  which divides $U_{\mathcal{I}}$ and 
  one of $\{ z, x_{{\ell}_2} \}$, $\{ z, y_{{\ell}_2} \}$ is an edge of $G$. 
  Similarly, we define $V_3, \ldots, V_r$. 
  Note that $x_{{\ell}_k}, y_{{\ell}_k} \in V_k$ and 
  $V_1, \ldots, V_r$ are pairwise disjoint. 
  Then it is enough to show that $\sigma = V_1 \cup \cdots \cup V_r$ and 
  $G_{V_{k}}$ is a complete bipartite subgraph of $G$. 

  \par
  We first show that $\sigma = V_1 \cup \cdots \cup V_r$. 
  It is clear that 
  \begin{displaymath}
    \{ x_{{\ell}_1}, \ldots, x_{{\ell}_r},\,  y_{{\ell}_1}, \ldots, y_{{\ell}_r} \} 
    \subset V_1 \cup \cdots \cup V_r \subset \sigma. 
  \end{displaymath}
  Put ${\sigma}_0 = \sigma \setminus 
        \{ x_{{\ell}_1}, \ldots, x_{{\ell}_r},\,  
           y_{{\ell}_1}, \ldots, y_{{\ell}_r} \}$. 
  For $x_{\ell} \in {\sigma}_0$, since $p_{\ell} \in \mathcal{I}$, 
  there exists ${\ell}_0 \in \{ {\ell}_1, \ldots, {\ell}_r \}$ 
  such that $p_{\ell} \leq p_{{\ell}_0}$. 
  Then $\{ x_{\ell}, y_{{\ell}_0} \} \in E(G)$. Therefore 
  $x_{\ell} \in V_1 \cup \cdots \cup V_{{\ell}_0}$. 
  Next we consider about $y_{{\ell}'} \in {\sigma}_0$. 
  In this case $p_{{\ell}'} \in P_G \setminus \mathcal{I}$. 
  If there is no ${\ell}_0 \in \{ {\ell}_1, \ldots, {\ell}_r \}$ such that 
  $p_{{\ell}_0} \leq p_{{\ell}'}$, then 
  $M(\mathcal{I}) \cup \{ p_{{\ell}'} \} 
   = \{ p_{{\ell}_1}, \ldots, p_{{\ell}_r}, p_{{\ell}'} \}$ is an antichain. 
  Let $\mathcal{I}'$ be the poset ideal generated by 
  $M(\mathcal{I}) \cup \{ p_{{\ell}'} \}$ 
  and set $T' = T \cup \{ p_{{\ell}'} \}$. 
  Then $\mathcal{I}' \cap T' = M(\mathcal{I}')$ and it follows that 
  $e(\mathcal{I}', T')$ is a free basis. 
  Therefore $[\mathcal{I}' \setminus M(\mathcal{I}'), \mathcal{I}']$ 
  is a Boolean sublattice of $\mathcal{L}_G$ by Lemma \ref{Boolean}. 
  Since 
  \begin{displaymath}
    [\mathcal{I} \setminus M(\mathcal{I}), \mathcal{I}] \subsetneq 
    [\mathcal{I}' \setminus M(\mathcal{I}'), \mathcal{I}'], 
  \end{displaymath}
  this contradicts to the maximality of 
  $[\mathcal{I} \setminus M(\mathcal{I}), \mathcal{I}]$. 
  Therefore there exists ${\ell}_0 \in \{ {\ell}_1, \ldots, {\ell}_r \}$ 
  such that 
  $p_{{\ell}_0} \leq p_{{\ell}'}$. When this is the case, 
  $\{ x_{{\ell}_0}, y_{{\ell}'} \} \in E(G)$ as required. 

  \par
  We next show that $G_{V_{k}}$ is a complete bipartite subgraph of $G$. 
  Take $x_{\ell}, y_{{\ell}'} \in V_{k}$. 
  Then from the construction of $V_{k}$, we have 
  $\{ x_{\ell}, y_{{\ell}_{k}} \}, \{ x_{{\ell}_k}, y_{{\ell}'} \} \in E(G)$. 
  Since $G$ satisfies (CM3), it follows that $\{ x_{\ell}, y_{{\ell}'} \}$ 
  as desired. 
\end{proof}

\section{Unmixed bipartite graphs}
\label{sec:UnmixedBipartite}
In this section we consider unmixed bipartite graphs $G$. 
In \cite{Kummini}, Kummini investigated the edge ideals of such graphs. 
He constructed the acyclic reduction $\widehat{G}$, 
which is a Cohen--Macaulay bipartite graph, 
from an unmixed bipartite graph $G$ 
and describe the regularity and the projective dimension of $I(G)$ 
in terms of the Alexander dual ideal of the edge ideal of $\widehat{G}$ 
(\cite[Proposition 3.2]{Kummini}). 
In particular, he proved $\reg (S/I(G)) = a(G)$. 
In this section, we focus on the projective dimension and 
give a characterization of it
as a generalization of Corollary \ref{CMbipartiteProj} by using 
Kummini's results. 

\par
Precisely, the following theorem is the main result in this section. 
\begin{theorem}
  \label{unmixedProj}
  Let $G$ be an unmixed bipartite graph. Then 
    \begin{displaymath}
    \pd (S/I(G)) = \max \left\{ \# V ({\mathcal{B}}) - r \; : \; 
    \begin{aligned}
      &\text{$\mathcal{B} = \{ B_1, \ldots, B_r \}$ 
            is a pairwise $3$-disjoint set} \\ 
      &\text{of complete bipartite subgraphs of $G$} 
    \end{aligned}
    \right\}
  \end{displaymath}
\end{theorem}

\par
First, we recall the Kummini's idea in \cite{Kummini}. 
Let $G$ be an unmixed bipartite graph on $V$ without isolated vertices. 
Then Villarreal \cite{Villarreal07} proved that $V$ can be bipartitioned as 
$V = \{ x_1, \ldots, x_n \} \cup \{ y_1, \ldots, y_n \}$ with 
the properties (CM1) and (CM3) in Theorem \ref{CMbipartite}. 
In particular, $G$ has a perfect matching. 
With a bipartite graph with perfect matching on the vertex set 
$V = \{ x_1, \ldots, x_n \} \cup \{ y_1, \ldots, y_n \}$, 
we can associate the directed graph ${\mathfrak{d}}_G$ on the vertex set 
$[n] := \{ 1, 2, \ldots, n \}$: $ij$ ($i \neq j$) is a (directed) edge of 
${\mathfrak{d}}_G$ if and only if $\{ x_i, y_j \}$ is an edge of $G$. 
When $G$ is unmixed, it follows by (CM3) that the corresponding directed graph 
${\mathfrak{d}}_G$ is transitive, that is, both $ij$ and $jk$ are edges of 
${\mathfrak{d}}_G$, then $ik$ is also an edge of ${\mathfrak{d}}_G$. 
Moreover when $G$ is Cohen--Macaulay, it follows by (CM2) that 
the corresponding directed graph ${\mathfrak{d}}_G$ is acyclic, that is, 
there is no directed cycle. 

\par
Let $G$ be an unmixed bipartite graph and 
$\mathfrak{d} = {\mathfrak{d}}_G$ the corresponding directed graph on $[n]$. 
A pair $i, j$ of vertices of $\mathfrak{d}$ are said to be 
\textit{strongly connected} if both of $ij$ and $ji$ 
are edges of $\mathfrak{d}$. 
Then strongly connected components form a partition of the vertex set. 
Let ${\mathcal{Z}}_1, \ldots, {\mathcal{Z}}_t$ be strongly connected 
components of $\mathfrak{d}$. 
We define the directed graph $\widehat{\mathfrak{d}}$ on $[t]$ 
by setting $ab$ ($a \neq b$) is a (directed) edge of $\widehat{\mathfrak{d}}$ 
if and only if $ij$ is an edge of $\mathfrak{d}$ 
for some (all) $i \in {\mathcal{Z}}_a$ and some (all) 
$j \in {\mathcal{Z}}_b$. 
Then $\widehat{\mathfrak{d}}$ is acyclic. 
Also it is transitive since $\mathfrak{d}$ is transitive. 
Let $\widehat{G}$ is the bipartite graph on 
$\{ u_1, \ldots, u_t \} \cup \{ v_1, \ldots, v_t \}$ such that 
$\{ u_a, v_a \}$ is an edge of $\widehat{G}$ for $a = 1, \ldots, t$, and 
for $a \neq b$, 
$\{ u_a, v_b \}$ is an edge of $\widehat{G}$ if and only if 
$ab$ is an edge of $\widehat{\mathfrak{d}}$. 
Then $\widehat{G}$ is a Cohen--Macaulay bipartite graph. 
We call $\widehat{G}$ the \textit{acyclic reduction} of $G$. 
We set ${\zeta}_a = \# {\mathcal{Z}}_a$ for $a=1, \ldots, t$. 
Also for $\sigma = \prod_a u_a^{{s}_a} \prod_b v_b^{{r}_b}$, 
we set 
${\sigma}^{\zeta} = \prod_a u_a^{s_a {\zeta}_a} \prod_b v_b^{r_b {\zeta}_b}$. 

\par
Kummini \cite{Kummini} proved the following proposition. 
\begin{proposition}[{Kummini \cite[Proposition 3.2]{Kummini}}]
  \label{Kummini}
  Let $G$ be an unmixed bipartite graph. 
  Then 
  \begin{displaymath}
    \pd (S/I(G)) = \max \{ |{\sigma}^{\zeta}| - r \; : \; 
      \beta_{r, \sigma} ((I(\widehat{G}))^{\ast}) \neq 0 \}. 
  \end{displaymath}
\end{proposition}

\par
Now we prove Theorem \ref{unmixedProj}. 
\begin{proof}[Proof of Theorem \ref{unmixedProj}]
  Take $\beta_{r, \sigma} ((I(\widehat{G}))^{\ast}) \neq 0$ which gives 
  $\pd (S/I(G))$. 
  We first prove that we may assume that 
  $\beta_{r, \sigma} ((I(\widehat{G}))^{\ast}) \neq 0$ is extremal. 
  
  \par
  Suppose that $\beta_{r, \sigma} ((I(\widehat{G}))^{\ast})$ 
  is not extremal. Then there exists 
  $\beta_{s, \tau} ((I(\widehat{G}))^{\ast}) \neq 0$ with 
  $s \geq r$, $\tau \succ \sigma$, and $|{\tau}| - |{\sigma}| \geq s-r$. 
  Note that both $\sigma$ and $\tau$ are $(0,1)$-vectors since 
  $(I(\widehat{G}))^{\ast}$ is a squarefree monomial ideal. 
  Thus we identify $\sigma$ and $\tau$ with subsets of the vertex set of 
  $\widehat{G}$ as before. 
  Then $\tau \succ \sigma$ implies $\tau \supsetneq \sigma$ and we have 
  \begin{displaymath}
    (|{\tau}^{\zeta}| - s) - (|{\sigma}^{\zeta}| - r) 
    = \sum_{u_{a} \in \tau \setminus \sigma} {\zeta}_{a} 
      + \sum_{v_{b} \in \tau \setminus \sigma} {\zeta}_b 
      - (s - r)
    \geq (|{\tau}| - |{\sigma}|) - (s-r) \geq 0. 
  \end{displaymath}
  Therefore we can replace 
  $\beta_{r, \sigma} ((I(\widehat{G}))^{\ast})$ 
  by $\beta_{s, \tau} ((I(\widehat{G}))^{\ast})$. 

  \par
  Let $\beta_{r, \sigma} ((I(\widehat{G}))^{\ast}) \neq 0$ be an 
  extremal Betti number which gives $\pd (S/I(G))$. 
  In order to prove the theorem, 
  it is sufficient to construct a pairwise $3$-disjoint set 
  $\mathcal{B} = \{ B_1, \ldots, B_r \}$ of complete bipartite subgraphs 
  of $G$ with $\# V(\mathcal{B}) = |{\sigma}^{\zeta}|$. 
  Since $\widehat{G}$ is Cohen--Macaulay bipartite graph, 
  there exists a pairwise $3$-disjoint set 
  $\widehat{\mathcal{B}} = \{ \widehat{B}_1, \ldots, \widehat{B}_r \}$ of 
  complete bipartite subgraphs of $\widehat{G}$ 
  by Proposition \ref{extrem-CMbipartite}. 
  Let $B_k$ be the complete bipartite graph with the vertex partition 
  \begin{displaymath}
      V(B_k) = 
      \left( \bigcup_{u_a \in V(\widehat{B}_k)} \{ x_p \; : \; p \in \mathcal{Z}_a \} \right) 
      \sqcup 
      \left( \bigcup_{v_b \in V(\widehat{B}_k)} \{ y_q \; : \; q \in \mathcal{Z}_b \} \right). 
  \end{displaymath}

  \par
  We first prove that $B_k$ is a subgraph of $G$, that is, 
  we prove that $\{ x_p, y_q \}$ is an edge of $G$ for 
  $p \in \mathcal{Z}_a$, $q \in \mathcal{Z}_b$ 
  with $u_a \in V(\widehat{B}_k), v_b \in V(\widehat{B}_k)$. 
  If $a=b$, then $p, q$ belongs to the same strongly connected component 
  of $\mathfrak{d}$. 
  Therefore $\{ x_p, y_q \} \in E(G)$. 
  Assume $a \neq b$. Since $\widehat{B}_k$ is a complete bipartite subgraph 
  of $\widehat{G}$, 
  $\{ u_a, v_b \}$ is an edge of $\widehat{G}$. Therefore 
  $ab$ is an edge of $\widehat{\mathfrak{d}}$. 
  Then it follows that $pq$ is an edge of $\mathfrak{d}$ and 
  that $\{ x_p, y_q \} \in E(G)$. 

  \par
  Since $\widehat{\mathcal{B}}$ is pairwise $3$-disjoint, 
  we can choose $\widehat{e}_k = \{ u_{a_k}, v_{b_k} \} \in E(\widehat{B}_k)$, 
  $k = 1, \ldots r$ 
  those are pairwise $3$-disjoint. 
  Take $p_k \in \mathcal{Z}_{a_k}$ and $q_k \in \mathcal{Z}_{b_k}$ for each $k$. 
  Then $e_k := \{ x_{p_k}, y_{q_k} \}$ is an edge of $B_k$. 
  If $\{ x_{p_k}, y_{q_{\ell}} \}$ is an edge of $G$ for $k \neq \ell$, 
  then $p_k q_{\ell}$ is an edge of $\mathfrak{d}$. 
  Note that $p_k \in \mathcal{Z}_{a_k}$, $q_{\ell} \in \mathcal{Z}_{b_{\ell}}$. 
  If $a_k = b_{\ell}$, then $\widehat{e}_k = \{ u_{b_{\ell}}, v_{b_k} \}$ and 
  $\widehat{e}_{\ell} = \{ u_{a_{\ell}}, v_{b_{\ell}} \}$. 
  This contradicts to the $3$-disjointness of $\widehat{e}_k$ and 
  $\widehat{e}_{\ell}$ because $\{ u_{b_{\ell}}, v_{b_{\ell}} \}$ is an edge 
  of $\widehat{G}$. 
  When $a_k \neq b_{\ell}$, it follows that $a_k b_{\ell}$ is an edge of 
  $\widehat{\mathfrak{d}}$. In particular, $\{ u_{a_k}, v_{b_{\ell}} \}$ 
  is an edge of $\widehat{G}$. This also contradicts to 
  the $3$-disjointness of $\widehat{e}_k$ and $\widehat{e}_{\ell}$. 
  Therefore we conclude that $e_1, \ldots, e_r$ are pairwise $3$-disjoint. 

  \par
  Since $V(\widehat{B}_k) \cap V(\widehat{B}_{\ell}) = \emptyset$, 
  we have $V(B_k) \cap V(B_{\ell}) = \emptyset$ for $k \neq \ell$. 
  Therefore $\mathcal{B} = \{ B_1, \ldots, B_r \}$ is a pairwise $3$-disjoint 
  set of complete bipartite subgraphs of $G$. 
  Since $|{\sigma}^{\zeta}| = \# V(\mathcal{B})$, 
  the assertion follows. 
\end{proof}

\begin{acknowledgement}
  The author thanks the referee for reading the manuscript carefully. 

  \par
  The author is partially supported 
  by JSPS Grant-in-Aid for Young Scientists (B) 24740008. 
\end{acknowledgement}


\end{document}